\documentclass[a4paper,12pt, reqno]{amsart}
\usepackage{amsfonts, color}
\usepackage[textwidth=460pt]{geometry}
\usepackage{amstext, amsthm, amssymb, amsmath}
\usepackage{mathtools}
\usepackage{mathrsfs}
\usepackage{graphicx}
\usepackage{color}
\usepackage[linesnumbered,lined,ruled,commentsnumbered]{algorithm2e}

\usepackage{hyperref}

\usepackage{cleveref}
\hypersetup{colorlinks = true, citecolor = blue}

\theoremstyle{plain}
\begingroup
\newtheorem{theorem}{Theorem}[section]
\newtheorem{lemma}[theorem]{Lemma}
\newtheorem{proposition}[theorem]{Proposition}
\newtheorem{corollary}[theorem]{Corollary}
\newtheorem{conjecture}{Conjecture}
\endgroup

\theoremstyle{definition}
\begingroup
\newtheorem{defn}{Definition}

\endgroup

\theoremstyle{remark}
\newtheorem{remark}{Remark}

\numberwithin{equation}{section}
\DeclareMathOperator*{\argmin}{arg\,min}

\newcommand{\R}{\mathbb{R}}

\newcommand{\e}{\varepsilon}

\newcommand{\NN}{\mathbb{N}}

\newcommand{\Rext}{\mathbb{R}\cup\{+\infty\}}
\newcommand{\RextM}{\mathbb{R}\cup\{-\infty\}}
\newcommand{\val}{\mathcal{H}}

\DeclareMathAlphabet{\mathbbb}{U}{bbold}{m}{n}

\crefname{theorem}{Theorem}{Theorems}
\crefname{lemma}{Lemma}{Lemmas}
\crefname{corollary}{Corollary}{Corollaries}
\crefname{section}{Section}{Sections}
\crefname{proposition}{Proposition}{Proposition}
\crefname{defn}{Definition}{Definitions}
\crefname{remark}{Remark}{Remarks}
\crefname{conjecture}{Conjecture}{Conjectures}

\usepackage[foot]{amsaddr}

\title{Trade-off Invariance principle for minimizers of regularized functionals}

\author{Massimo Fornasier$^{1, 2, 3}$}
\email{massimo.fornasier@cit.tum.de}

\author{Jona Klemenc$^{1,2}$}
\email{jona.klemenc@tum.de}
\author{Alessandro Scagliotti$^{1,2}$}
\email{scag@ma.tum.de }

\address{$^1$CIT School, Technical University of Munich, Garching bei M\"unchen, Germany.}
\address{$^2$Munich Center for Machine Learning (MCML), Munich, Germany.}
\address{$^3$Munich Data Science Institute (MDSI), Munich, Germany}

\date{\today}

\begin{document}

\begin{abstract}
    In this paper, we consider functionals of the form  $H_\alpha(u)=F(u)+\alpha G(u)$ with $\alpha\in[0,+\infty)$, where $u$ varies in a set $U\neq\emptyset$ (without further structure).
    We first revisit a result stating that, excluding at most countably many values of $\alpha$, we have $\inf_{H_\alpha^\star}G= \sup_{H_\alpha^\star}G$, where $H_\alpha^\star \coloneqq \arg\min_UH_\alpha$, which is assumed to be non-empty.
    Then, we prove a stronger result that concerns the invariance of the limiting value of the functional $G$ along minimizing sequences for $H_\alpha$, which extends the above Principle to the case $H_\alpha^\star= \emptyset$.
    Moreover, we show to what extent these findings generalize to multi-regularized functionals and---in the presence of an underlying differentiable structure---to critical points.
    Finally, the main result implies an unexpected consequence for functionals regularized with uniformly convex norms: excluding again at most countably many values of $\alpha$, it turns out that for a minimizing sequence, convergence to a minimizer in the weak or strong sense is equivalent.
\end{abstract}

\maketitle



\section{Introduction}\label{sec:intro}
Minimization problems for functionals $H_\alpha\colon U\to\Rext$ of the form
\begin{equation}\label{eq:def_H}
    H_\alpha(u)\coloneqq F(u)+\alpha G(u),
\end{equation}
where $\alpha\in[0,+\infty)$, $U\neq\emptyset$, $F, G\colon U\to\Rext$ and $F+G\not \equiv +\infty$,
are ubiquitous in Inverse Problems~\cite{InverseProblems,tikhonov1998} and have enormous relevance for applications.
Without further assumptions on the involved objects---such as convexity of $H_\alpha$, which is rather difficult to have in practice---it is common wisdom that there is little we can say about the structure of the set of minimizers.
In particular, there might exist distinct minimizers $u^\star_1 \neq u^\star_2$.
In this case, despite the fact that both $u^\star_1$ and $u^\star_2$ achieve the infimum of $H_\alpha$,
we could expect them to attain distinct trade-offs between the competing terms $F$ and $G$.
In other words, we could expect that $F(u^\star_1) \neq F(u^\star_2)$ and that $G(u^\star_1) \neq G(u^\star_2)$.
However, in this paper, we shall show that such a variance of trade-offs between minimizers is exceptional, in the sense that it can only occur at a countable values of $\alpha$. We start with the simplest result, which applies to the situation described above.
\begin{theorem}[Trade-off Invariance Principle I]\label{thm:principle}
    Let $H_\alpha\colon U\to\Rext$ be defined as in~\eqref{eq:def_H} and let us define
    \begin{equation} \label{eq:def_argmin}
        U\supset H_\alpha^\star\coloneqq \arg\min_{u\in U} H_\alpha(u).
    \end{equation}
    If there exist $a<b\in[0,+\infty)$ such that $H_\alpha^\star \neq \emptyset$ for every $\alpha\in [a,b]$, then
    for all but countable $\alpha$ in $[a,b]$
    we have that
    \begin{equation} \label{eq:principle}
        \inf_{u\in H_\alpha^\star} G(u) =
        \sup_{u\in H_\alpha^\star} G(u).
    \end{equation}
\end{theorem}
\begin{remark}\label{rem:counterexample}
    As detailed in the proof, the values of $\alpha$ for which the conclusion of \cref{thm:principle} does not hold are included in the set of discontinuities of a monotone non-increasing function. Hence, it follows that they are at most countably many.
    Moreover, it is possible to construct examples where there are actually infinitely many instances of such values of $\alpha$.
    Namely, let us consider $U=[0,+\infty)$, $F(u)=\sum_{i=1}^\infty 2^{-i}(u-i+1)_+$ and $G(u)=-u$, where $(z)_+$ denotes the positive part of $z\in\R$.
    Then, setting $H_\alpha\coloneqq F+\alpha G$ with $\alpha\in [0,+\infty)$, it turns out that
    \begin{equation*}
        H^\star_\alpha =
        \begin{cases}
            \{0\}   & \alpha=\alpha^0,                   \\
            [k-1,k] & \alpha=\alpha^k,                   \\
            \{k\}   & \alpha\in (\alpha^k,\alpha^{k+1}), \\
        \end{cases}
    \end{equation*}
    where $\alpha^0\coloneqq 0$ and $\alpha^k\coloneqq \sum_{i=1}^k2^{-i}$ for every $k\geq1$. In particular, when $\alpha=\alpha^k$ for $k\geq1$ we obtain $-k = \inf_{u\in H^\star_\alpha}G(u) < \sup_{u\in H^\star_\alpha}G(u) = -k+1$. Finally, we observe that for $\alpha\in [1,+\infty)$ we have that $H_\alpha^\star = \emptyset$.
\end{remark}

Results in the vein of \cref{thm:principle} have been already observed in the Inverse Problem literature (see \cite[Lemma~2.6.3]{tikhonov1998}).
For instance, in \cite{ito2011aregularization} the authors were interested in  tuning the Tikhonov regularization parameter $\alpha$ according to noise level to obtain rates of convergence to true solutions. In the same paper, they related the identity in \eqref{eq:principle} to the points $\alpha\in[0,+\infty)$ where the concave `value function' $\mathcal{H}\colon [0,+\infty)\to \R\cup\{-\infty\}$ is differentiable, having set $\mathcal{H}(\alpha)\coloneqq \inf_U H_\alpha$ (see \cite[Theorem~2.1 and Corollary~2.1]{ito2011aregularization}).
Because their initial motivation came from inverse problems, in \cite{ito2011aregularization, ito2011multiparameter} the authors assumed Banach space structure for the domain $U$ and non-negativity for the regularizer $G$ (for the sake of the precision, in \cite{ito2011multiparameter} convexity and weak lower semi-continuity of the regularizer are further required). Yet, all these hypotheses are in fact not necessary, as made explicit in our formulation of Theorem \ref{thm:principle}.

Besides the immediate generalizations just mentioned, the main advancement of the present paper
concerns the extension of \cref{thm:principle} to the limiting values of $F$ and
$G$ along minimizing sequences.
In particular, the following extension of our argument allows us to avoid
the assumption that $H^\star_\alpha \neq \emptyset$, which is present in \cite{ito2011multiparameter, ito2011aregularization, tikhonov1998} and \cref{thm:principle}:

\begin{theorem}[Trade-off Invariance Principle II]\label{thm:principle_relaxed}
    For all but countable $\alpha \in [0,+\infty)$,
    if $\inf_{u\in U} H_{\alpha}(u) > -\infty$, then there exists
    some $G_\alpha \in [-\infty, +\infty]$ such that $G(u_i) \rightarrow G_\alpha$
    for every minimizing sequence $(u_i)_{i \in \NN}$
    of $H_\alpha$.
\end{theorem}
As a relevant and, perhaps, unexpected example of applying the Trade-off Invariance Principle, we propose the next corollary. Let us first introduce the following notion:
\begin{defn} \label{def:weak2strong}
    Let $U$ be a  Banach space. A functional $G:U \to \R$ is called a \emph{weak-to-strong} functional if for every sequence $(u_i)_{i \in  \NN}$ weakly converging to some $u^\star$ such that $\lim_{i\to +\infty} G(u_i) = G(u^\star)$, then the sequence converges also strongly.
\end{defn}
Widely used instances of weak-to-strong functionals are norms in uniformly convex Banach spaces~\cite[Proposition~3.32]{B11}, or  more general strictly convex functionals as in~\cite{Vis84}.

\begin{corollary}\label{cor:strong_convergence}
    Let $U$ be a  Banach space, and let us consider $F\colon U \to \Rext$ such that $F\not\equiv +\infty$ and $G\colon U \to \R$ a \emph{weak-to-strong} functional.
    Then, for all but countable $\alpha \in (0, +\infty)$, the following holds: \\
    \noindent    If $(u_i)_{i \in \NN}$ is a minimizing sequence of the functional $H_\alpha \coloneqq F + \alpha G$,
    which converges \emph{weakly} to some minimizer $u^\star \in U$ of $H_\alpha$,
    then $(u_i)_{i \in \NN}$ converges \emph{strongly} to $u^\star$.
\end{corollary}
\begin{proof}
    Since $F\not\equiv +\infty$, we have that $F+G \not \equiv +\infty$
    Let $\alpha\in (0,+\infty)$ be such that the conclusion of \cref{thm:principle_relaxed} holds --- we recall that the complement of this set in $(0,+\infty)$ is at most countable.
    Then, let us consider a minimizing sequence $(u_i)_{i \in \NN}$ for $H_\alpha$ that is converging weakly to $u^\star$, as in the statement.
    Since by hypothesis $\inf_u H_\alpha(u) = H_\alpha(u^\star)$, we can apply \cref{thm:principle_relaxed} to the constant sequence $(u_i')_{i \in \NN}$ with $u'_i=u^\star$ for every $i\in\NN$ and to $(u_i)_{i \in \NN}$
    to deduce that $\lim_{i \to \infty} G(u_i) = G_\alpha = G(u^\star)$.
    Recalling that $G$ is assumed to be a weak-to-strong functional (see \cref{def:weak2strong}), the thesis follows.
\end{proof}

\begin{remark}
    If we further assume that $H_\alpha$ is sequentially weakly lower semi-continuous, then \cref{cor:strong_convergence} implies that every weakly pre-compact minimizing sequence is actually strongly pre-compact for every but countably many $\alpha\in(0,+\infty)$.
    In certain applications, obtaining strong convergence is beneficial in order to lift properties of the minimizing sequence to its limiting minimizer. 
\end{remark}

\begin{remark}
    {Usually strongly convergent minimizing sequences can be obtained when the functionals are convex by means of Mazur's Lemma~\cite[Corollary~3.8]{B11}. The previous result is significantly stronger as it applies to more general functionals $F$ (also nonconvex), and it holds for any weakly convergent minimizing sequence. It is worth noticing that the sum $F+\alpha G$ is not required to be convex either.}
\end{remark}

\begin{remark}
    In some very special cases, the conclusion of \cref{cor:strong_convergence} was already accessible.
    Namely, for $F$ being the indicator function of a closed convex set $C$ in a Hilbert space $U$ and $G(u)=\|u-u_0\|_U$ being the distance from a given $u_0$,
    then the minimizer of $H_\alpha$ is the orthogonal projection of $u_0$ onto $C$.
    In this case \emph{any minimizing sequence}  of $H_\alpha$ is actually a Cauchy sequence as in the proof presented in~\cite[Theorem~5.2]{B11}.
\end{remark}

\begin{remark}
    Our initial interest in \cref{thm:principle} 
    was sparked by the minimizing movement scheme (see~\cite{greenbook}). Namely, in~\cite[Lemma 3.1.2]{greenbook}, for a metric space $U$ equipped with a distance $d$, for a functional $\phi\colon U \rightarrow \R$ and for a fixed $u \in U$, the authors consider the minimizing set
    \begin{equation*}
        J_\tau[u] \coloneqq \argmin_{v \in U} \left\{\phi(v) + \frac{1}{2\tau}d^2(v, u) \right\}.
    \end{equation*}
    Under the assumption that $J_\tau[u] \neq \emptyset$ for $\tau \in (0, \tau_\circ]$ for some $\tau_\circ$,
    they show that, for all but countable $\tau \in (0, \tau_\circ]$,
    \begin{equation*}
        \inf_{u_\tau \in J_\tau[u]} d(u_\tau, u) = \sup_{u_\tau \in J_\tau[u]} d(u_\tau, u).
    \end{equation*}
\end{remark}

In the following, we first give a simple proof of \cref{thm:principle} in \cref{sec:principle_proof}.
In \cref{sec:principle_relaxed_proof}, we show a slightly more technical, but still elementary proof of \cref{thm:principle_relaxed}. Then, we use \cref{thm:principle_relaxed} to prove a refined version of
\cref{thm:principle} stated in \cref{thm:principle_refined}, which shows that \cref{thm:principle_relaxed} is indeed a generalization of \cref{thm:principle}.
Moreover, we present in \cref{sec:multi} the extension of the Principle to multi-regularized functionals (see \cref{thm:multi_reg}), generalizing the results of \cite{ito2011multiparameter} that required the existence of minimizers \cite[Assumption~2.1]{ito2011multiparameter}.
In \cref{sec:critical}, assuming that the functional $H_\alpha$ is differentiable and satisfies the local {\L}ojasiewicz inequality, we generalize the Principle to the critical points of $H_\alpha$ in \cref{thm:crit_points}.
Finally, in \cref{sec:consequences} we collect some interesting applications of the main findings of this work.
Among others, we report that in \cref{subsec:diff_value} we managed to prove that the value function $\val$ is differentiable at $\alpha$ if and only if the Trade-off Invariance Principle for minimizing sequences holds at $\alpha$.
This shows that the results contained in \cite{ito2011multiparameter, ito2011aregularization} can be established also without the assumption $\arg\min_U H_\alpha \neq \emptyset$.


\section{Derivation of \texorpdfstring{\cref{thm:principle}}{Theorem~\ref{thm:principle}}}\label{sec:principle_proof}

We first establish a key-lemma.
We report that it can be found in \cite[Lemma~2.6.1]{tikhonov1998}, but we present its proof for the sake of completeness.

\begin{lemma}[Monotonicity Lemma I]\label{lem:monotone_G}
    Let $H_\alpha\colon U\to\Rext$ and $H_\alpha^\star$ be defined as in~\eqref{eq:def_H} and~\eqref{eq:def_argmin}, respectively, and let us consider $\alpha_2 > \alpha_1$ such that $H_{\alpha_i}^\star \neq \emptyset$ for $i=1,2$.
    Then, we have that
    \begin{equation} \label{eq:monotone_G}
        \inf_{u\in H_{\alpha_1}^\star} G(u) \geq
        \sup_{u\in H_{\alpha_2}^\star} G(u).
    \end{equation}
\end{lemma}
\begin{proof}
    Let us consider $u_{\alpha_i}\in H_{\alpha_i}^\star$ for $i=1,2$.
    Then, observing that $H_{\alpha_2}(u_{\alpha_2}) \leq H_{\alpha_2}(u_{\alpha_1})$, we deduce that
    \begin{equation*}
        F(u_{\alpha_2})+\alpha_2 G(u_{\alpha_2})
        \leq F(u_{\alpha_1})+\alpha_2 G(u_{\alpha_1}),
    \end{equation*}
    which yields
    \begin{equation} \label{eq:aux_ineq}
        F(u_{\alpha_2})- F(u_{\alpha_1})
        \leq \alpha_2 \big( G(u_{\alpha_1})
        - G(u_{\alpha_2}) \big).
    \end{equation}
    Let us assume by contradiction that $G(u_{\alpha_2})>G(u_{\alpha_1})$, resulting in $G(u_{\alpha_1}) - G(u_{\alpha_2})< 0$.
    In this case, we would get $H_{\alpha_1}(u_{\alpha_2})< H_{\alpha_1}(u_{\alpha_1})$, which contradicts the minimality of $u_{\alpha_1}$ for $H_{\alpha_1}$.
    Namely, using~\eqref{eq:aux_ineq}, it follows that
    \begin{equation*}
        \begin{split}
            H_{\alpha_1}(u_{\alpha_2}) & =
            F(u_{\alpha_2})+\alpha_1 G(u_{\alpha_2})                                                               \\
                                       & \leq F(u_{\alpha_1})+\alpha_1 G(u_{\alpha_2}) +
            \alpha_2 \big( G(u_{\alpha_1})
            - G(u_{\alpha_2}) \big)                                                                                \\
                                       & = F(u_{\alpha_1})+\alpha_1 G(u_{\alpha_1}) +
            (\alpha_2 -\alpha_1) \big( G(u_{\alpha_1})
            - G(u_{\alpha_2}) \big)                                                                                \\
                                       & = H_{\alpha_1}(u_{\alpha_1}) + (\alpha_2 -\alpha_1) \big( G(u_{\alpha_1})
            - G(u_{\alpha_2})\big)                                                                                 \\
                                       & < H_{\alpha_1}(u_{\alpha_1}),
        \end{split}
    \end{equation*}
    which is impossible. Therefore, we deduce that
    \begin{equation*}
        G(u_{\alpha_2}) \leq G(u_{\alpha_1}) \quad \forall \,u_{\alpha_1}\in H^\star_{\alpha_1}, \forall\, u_{\alpha_2}\in H^\star_{\alpha_2},
    \end{equation*}
    and hence~\eqref{eq:monotone_G} follows.
\end{proof}

We are now in position to prove Theorem~\ref{thm:principle}.

\begin{proof}[Proof of Theorem~\ref{thm:principle}]
    Let us define $R\colon [a,b]\to\Rext$ as follows:
    \begin{equation}\label{eq:def_mon_funct}
        R(\alpha) \coloneqq \inf_{u\in H^\star_{\alpha}} G(u).
    \end{equation}
    Given $\alpha_2\geq \alpha_1$, we have that $R(\alpha_2)\leq R(\alpha_1)$. Indeed, in the interesting case $\alpha_2 > \alpha_1$, we can take advantage of Lemma~\ref{lem:monotone_G} to deduce that
    \begin{equation} \label{eq:mon_R}
        R(\alpha_1) = \inf_{u\in H^\star_{\alpha_1}} G(u) \geq
        \sup_{u\in H^\star_{\alpha_2}} G(u)
        \geq
        \inf_{u\in H^\star_{\alpha_2}} G(u) = R(\alpha_2).
    \end{equation}
    Let $\bar \alpha \in [a,b]$ be such that $\sup_{u\in H^\star_{\bar \alpha}} G(u)
        >
        \inf_{u\in H^\star_{\bar \alpha}} G(u)$.
    Then, from~\eqref{eq:mon_R} it follows that $\alpha$ must be a discontinuity point for the monotone decreasing function $R$. Indeed, we have that
    \begin{equation*}
        \lim_{\alpha\to\bar\alpha^-} R(\alpha) \geq \sup_{u\in H^\star_{\bar \alpha}} G(u) >
        \inf_{u\in H^\star_{\bar \alpha}} G(u) = R(\bar\alpha).
    \end{equation*}
    Since any monotone function admits at most countably many discontinuity points, we deduce that the identity~\eqref{eq:principle} is violated for at most countably many values of $\alpha$ in $[a,b]$.
\end{proof}
\section{Derivation of \texorpdfstring{\cref{thm:principle_relaxed}}{Theorem~\ref{thm:principle_relaxed}}}\label{sec:principle_relaxed_proof}

This section contains the first main contribution of the present paper. Indeed, we provide an extension of the results presented in \cite[Section~2.6]{tikhonov1998}, which required $\arg\min_U H_\alpha \neq \emptyset$ and, in some statements, the sequential coercivity and lower semi-continuity of the functionals $F,G$.

Before proving the theorem, we first establish the cornerstone of our argument.
\begin{lemma}[Monotonicity Lemma II]\label{lem:monotone_G_relaxed}
    For $\alpha \in [0, \infty)$, we define
    \begin{equation*}
        S_\alpha \coloneqq \{(u_i)_{i \in \NN} \text{ \emph{minimizing sequence of} } H_\alpha\},
    \end{equation*}
    and we set
    \begin{equation}\label{eq:g_definition}
        \begin{split}
            G_\alpha^+ & \coloneqq \sup_{(u_i) \in S_\alpha} \limsup_{i \to \infty} G(u_i) \in [-\infty, \infty], \\
            G_\alpha^- & \coloneqq \inf_{(u_i) \in S_\alpha} \liminf_{i \to \infty} G(u_i) \in [-\infty, \infty].
        \end{split}
    \end{equation}
    Then, for all $0 \leq \alpha_1 < \alpha_2 < \infty$ such that $\inf_u H_{\alpha_1}(u) > -\infty$
    and $\inf_u H_{\alpha_2}(u) > -\infty$,
    we have that
    \begin{equation}\label{eq:monotone_G_relaxed}
        G_{\alpha_1}^+ \geq G_{\alpha_1}^- \geq G_{\alpha_2}^+ \geq G_{\alpha_2}^-.
    \end{equation}
\end{lemma}
\begin{proof}
    The left and the right inequality are immediate from the definitions, so we only show the middle inequality.
    We assume, by contradiction, that for some $\alpha_1 < \alpha_2$ we have that $G_{\alpha_1}^- < G_{\alpha_2}^+$,
    and we define
    \begin{equation*}
        \epsilon \coloneqq \frac{G_{\alpha_2}^+ - G_{\alpha_1}^-}{5} > 0.
    \end{equation*}
    By the definition of $G_{\alpha_1}^-$ and $G_{\alpha_2}^+$,
    we can choose sequences $(u_j^1)_{j\in\NN} \in S_{\alpha_1}$ and $(u_k^2)_{k\in\NN} \in S_{\alpha_2}$ such that
    \begin{equation}\label{eq:limsup_diff}
        \limsup_{k \to \infty} G(u_k^2) - \liminf_{j \to \infty} G(u_j^1) \geq 4\epsilon.
    \end{equation}
    Since $F+G\not\equiv +\infty$, we have that $\inf_{u \in U} H_\alpha(u) < +\infty$ for all
    $\alpha \in [0, +\infty)$, and by assumption we also have that $\inf_u H_{\alpha_1}(u) > -\infty$ and $\inf_u H_{\alpha_2}(u) > -\infty$.
    Thus, the minimizing sequences $(H_{\alpha_1}(u_j^1))_{j\in\NN}$ and $(H_{\alpha_2}(u_k^2))_{k\in\NN}$ of $H_{\alpha_1}$ and $H_{\alpha_2}$, respectively, have finite limits.
    Setting $\Delta_\alpha \coloneqq \alpha_2 - \alpha_1 > 0$, for $j,k$ large enough we hence have that $H_{\alpha_1}(u_j^1)  \leq \inf_U H_{\alpha_1} + \Delta_\alpha\epsilon$ and $H_{\alpha_2}(u_k^2)  \leq \inf_U H_{\alpha_2} + \Delta_\alpha\epsilon$, which in particular implies
    \begin{alignat}{3}
        -H_{\alpha_1}(u_j^1) & \geq - &  & H_{\alpha_1}(u_k^2) - \Delta_\alpha \epsilon,\qquad\text{and}\label{eq:limdiff_1} \\
        H_{\alpha_2}(u_j^1)  & \geq   &  & H_{\alpha_2}(u_k^2) - \Delta_\alpha \epsilon.\label{eq:limdiff_2}
    \end{alignat}
    Furthermore, by \cref{eq:limsup_diff}, we can find arbitrarily large $j, k$ such that
    \begin{equation}\label{eq:G_diff}
        G(u_k^2) - G(u_j^1) \geq 3\epsilon.
    \end{equation}
    In particular, we can find $j, k$ such that \cref{eq:limdiff_1,eq:limdiff_2,eq:G_diff} hold simultaneously.
    Fixing such $j, k$, we add \cref{eq:limdiff_1,eq:limdiff_2} to get
    \begin{align*}
        \Delta_\alpha G(u_j^1) & = (H_{\alpha_2}(u_j^1) - H_{\alpha_1}(u_j^1))                                                                                  \\
                               & \geq  (H_{\alpha_2}(u_k^2) - H_{\alpha_1}(u_k^2)) - 2\Delta_\alpha\epsilon  = \Delta_\alpha G(u_k^2) - 2\Delta_\alpha\epsilon.
    \end{align*}
    Rearranging and using \cref{eq:G_diff}, we have
    \begin{equation*}
        2\Delta_\alpha\epsilon \geq \Delta_\alpha (G(u_k^2) - G(u_j^1)) \geq 3\Delta_\alpha\epsilon,
    \end{equation*}
    which is a contradiction and finishes the proof.
\end{proof}

We now prove an auxiliary result.

\begin{lemma}\label{lem:good_alpha_are_interval}
    Those $\alpha \in [0, \infty)$ for which $\inf_u H_{\alpha}(u) > -\infty$ form an interval in $[0, +\infty)$.
\end{lemma}
\begin{proof}
    We observe that this fact descends from the concavity of the value function $\val(\alpha)\coloneqq\inf_u  H_{\alpha}(u)$ (see \cref{lemma:inf_measurable}). However, we give here a direct proof.
    To prove the claim, it suffices to show that for $\alpha \in [0, +\infty)$,
    $\inf_u H_{\alpha}(u) = -\infty$ implies one of those two cases:
    \begin{enumerate}
        \renewcommand{\theenumi}{(\arabic{enumi})}
        \renewcommand{\labelenumi}{\theenumi}
        \item For all $\alpha' \geq \alpha$, $\inf_u H_{\alpha'}(u) = -\infty$.\label{case:bad_alphas_bigger}
        \item For all $\alpha' \leq \alpha$, $\inf_u H_{\alpha'}(u) = -\infty$.\label{case:bad_alphas_smaller}
    \end{enumerate}
    Let us assume that $\inf_u H_{\alpha}(u) = -\infty$ for some $\alpha \in [0, \infty)$,
    and let us take a sequence $(u_i)_{i \in \NN}$ such that $H_{\alpha}(u_i) \rightarrow -\infty$.
    In addition, there must exist one of following two subsequences:
    \begin{enumerate}
        \renewcommand{\theenumi}{(\alph{enumi})}
        \renewcommand{\labelenumi}{\theenumi}
        \item A subsequence $(u_{i_j})_{j \in \NN}$ such that $G(u_{i_j}) \leq 0$ for all $j \in \NN$.\label{case:g_small}
        \item A subsequence $(u_{i_j})_{j \in \NN}$ such that $G(u_{i_j}) \geq 0$ for all $j \in \NN$.\label{case:g_large}
    \end{enumerate}
    For $\alpha' \geq \alpha$,~\ref{case:g_small} implies that
    $H_{\alpha'}(u_{i_j}) \leq H_{\alpha}(u_{i_j}) \rightarrow -\infty$, and thus~\ref{case:g_small} implies~\ref{case:bad_alphas_bigger}.
    Likewise,~\ref{case:g_large} implies~\ref{case:bad_alphas_smaller}, and the claim follows.
\end{proof}

Finally, we proceed with the proof of \cref{thm:principle_relaxed}.

\begin{proof}[Proof of \cref{thm:principle_relaxed}]
    We must show that $G_\alpha$ exists for all $\alpha$ for which $\inf_u H_{\alpha}(u) > -\infty$, with at most
    countably many exceptions. Using \cref{lem:good_alpha_are_interval} and possibly neglecting the two endpoints
    ---as we allow for countable exceptions---, it suffices to prove the existence of such $G_\alpha$
    for all but countably many $\alpha$ in some open interval $(\tilde\alpha_1, \tilde\alpha_2)$, where $\tilde\alpha_1 \in [0, \infty)$, $\tilde\alpha_2 \in (\tilde\alpha_1, \infty]$ and
    $\inf_u H_{\alpha}(u) > -\infty$ for all $\alpha$ in $(\tilde\alpha_1, \tilde\alpha_2)$. \\
    By \cref{eq:monotone_G_relaxed}, the quantity $G_\alpha^+$ defined in \cref{eq:g_definition} is monotone decreasing on $(\tilde\alpha_1, \tilde\alpha_2)$ and for
    each $\alpha \in (\tilde\alpha_1, \tilde\alpha_2)$, $G_\alpha^- \neq G_\alpha^+$ implies that $\alpha$ is a discontinuity point of $G_\alpha^+$.
    Indeed, assume that $G_\alpha^- < G_\alpha^+$ for some $\alpha \in (\tilde\alpha_1, \tilde\alpha_2)$, and let
    $(\alpha_i)_{i \in \NN}$ be a sequence in $(\tilde\alpha_1, \tilde\alpha_2)$ converging to $\alpha$ from above.
    By \cref{eq:monotone_G_relaxed}, we have that
    \begin{equation*}
        \limsup_{i \to \infty} G_{\alpha_i}^+ \leq G_\alpha^- < G_\alpha^+,
    \end{equation*}
    and thus $\alpha$ is a discontinuity point of $G_\alpha^+$.
    Since a monotone function admits at most countably many discontinuity points, we deduce that $G_\alpha^- \neq G_\alpha^+$ for at most countably many values of $\alpha \in (\tilde\alpha_1, \tilde\alpha_2)$.
    Setting $G_\alpha \coloneqq G_\alpha^- = G_\alpha^+$ for all other $\alpha$ finishes the proof.
\end{proof}



We lastly prove that \cref{thm:principle_relaxed} implies a refined version of \cref{thm:principle},
which also shows that \cref{thm:principle_relaxed} is indeed a generalization of \cref{thm:principle}.
\begin{corollary}\label{thm:principle_refined}
    Let $H_\alpha:U\to\Rext$ and $ H_\alpha^\star\subset U$ be defined as in \cref{eq:def_H,eq:def_argmin}, respectively.
    Then, for all but countable $\alpha$ in $[0,+\infty)$,
    we have: if $H_\alpha^\star \neq \emptyset$, then
    \begin{equation*}
        \inf_{u\in H_\alpha^\star} G(u) =
        \sup_{u\in H_\alpha^\star} G(u).
    \end{equation*}
\end{corollary}
\begin{proof}
    If $H_\alpha^\star \neq \emptyset$, then for any $u^\star \in H_\alpha^\star$, we have $\inf_u H_\alpha(u) = H_\alpha(u^\star)>-\infty$,
    and we can apply \cref{thm:principle_relaxed} to find $G_\alpha$ for all but countable many values of $\alpha$.
    Using the constant minimizing sequence $(u_i')_{i \in \NN}$ with $u'_i=u^\star$, we deduce that
    $G(u^\star) = G_\alpha$ for every $u^\star \in H_\alpha^\star$. Hence, for all $\alpha$ where such a $G_\alpha$ exists, we have
    \begin{equation*}
        \inf_{u^\star\in H_\alpha^\star} G(u^\star) =
        G_\alpha =
        \sup_{u^\star\in H_\alpha^\star} G(u^\star). \qedhere
    \end{equation*}
\end{proof}

\section{Trade-off Invariance Principle for multi-regularization} \label{sec:multi}
In this section, we consider multi-regularized minimization problems.
To this end, we consider
a set $U \neq \emptyset$ and mappings $F, G_1, \ldots ,G_m\colon U \to \R \cup \{+\infty\}$ with $m\geq 1$.
For $\alpha =(\alpha_1, \ldots ,\alpha_m) \in [0, +\infty)^m$, we define $H_{\alpha}\colon U \to \R$ by
\begin{equation}\label{eq:multi_regularization}
    H_{\alpha}(u) \coloneqq F(u) + \sum_{j=1}^m \alpha_j G_j(u)
\end{equation}
for every $u \in U$.
As before, we assume the existence of some $\bar{u} \in U$
and some $\bar \alpha \in (0, +\infty)^m$ such that
\begin{equation}
    H_{\bar{\alpha}}(\bar{u}) < +\infty,
\end{equation}
so that $F + \sum_{j=1}^m G_j \not \equiv +\infty$, and $\inf_U H_\alpha < +\infty$ for every $\alpha \in [0, +\infty)^m$.
In this framework, we strive to formulate a version of the Trade-off Invariance Principle.
In other words, we try to establish a statement of the following form:
\begin{conjecture}\label{conj:multi_regularization}
    For all but some exceptional $\alpha  \in [0,+\infty)^m$,
    if $\inf_{u\in U} H_{\alpha}(u) > -\infty$, then there exist
    $G_\alpha^1,\ldots, G_\alpha^m \in [-\infty, +\infty]$ such that $G_j(u_i) \rightarrow G_\alpha^j$
    for every $j=1,\ldots,m$ and
    for every minimizing sequence $(u_i)_{i \in \NN}$
    of $H_\alpha$.
\end{conjecture}
The main obstacle is to describe the size of the set of exceptional values $\alpha$.
We note right away that, unlike in the single-regularization case, we cannot hope for the set
of exceptional pairs to be countable: Indeed, considering $m=2$ and adding the trivial regularizer $G_2(u) = 0$ to
the example in \cref{rem:counterexample}, all pairs of the form
$(\alpha^k, a)$ for $a \in [0, +\infty)$ are exceptional---and there are uncountably many of them. \\
On the other hand, constraining the set of exceptional pairs to be Lebesgue negligible seems
to be in reach: Let us denote by $B_1$ the set of those values $\alpha = (\alpha_1, \ldots ,\alpha_m)$ for which
$\inf_{u\in U} H_{\alpha}(u) > -\infty$, but for which we cannot find a $G_\alpha^1$ as above.
Then let us define $B_j$ analogously for $G_j$ with $j=2,\ldots,m$, i.e.,
\begin{equation} \label{eq:def_sets_Bj}
    B_j  \coloneqq \{\alpha  \in [0, +\infty)^m \mid G_{\alpha}^{j, +} > G_{\alpha}^{j, -} \,\, \mbox{  and } \,\, \mathrm{inf}_{u\in U} H_{\alpha}(u) > -\infty\}.
\end{equation}
If we assume for a moment that $B_1, \ldots,B_m$ are measurable,
then we can use Fubini-Tonelli's Theorem to estimate the size of the set of exceptional pairs:
\begin{align*}
    |B_1|  = \int_0^{+\infty} \cdots \int_0^{+\infty}\Bigg(\underbrace{\int_0^{+\infty} \mathbbb{1}_{B_1}(\alpha_1, \alpha_2) \, \mathrm{d}\alpha_1}_{=\, 0 \text{ by \cref{thm:principle_relaxed}}} \Bigg) \, \mathrm{d}\alpha_2 \cdots \mathrm{d}\alpha_m
    = 0.
\end{align*}
The same argument applies to $B_j$ with $j=2,\ldots,m$, and thus $|\bigcup_{j=1}^m B_j| = 0$.
However, it is a priori not obvious why the sets $B_j$ should be measurable. Indeed, it is worth recalling that we are not making any measurability assumption on the functions $F,G_1,\ldots,G_m$.\\
The aim of the rest of this section is to show that this is actually the case (see \cref{cor:B_measurable}). Therefore, using Fubini-Tonelli as mentioned above, we prove the following version of \cref{conj:multi_regularization}:
\begin{theorem} \label{thm:multi_reg}
    For Lebesgue-almost all $\alpha \in [0,+\infty)^m$,
    if $\inf_{u\in U} H_{\alpha}(u) > -\infty$, then there exist
    $G_\alpha^1,\ldots,  G_\alpha^m \in [-\infty, +\infty]$ such that $G_j(u_i) \rightarrow G_\alpha^j$  for every $j=1,\ldots,m$ and for every minimizing sequence $(u_i)_{i \in \NN}$
    of $H_\alpha$.
\end{theorem}

In analogy with the case of a single regularizer, we can deduce the following corollary.

\begin{corollary}
    Let $H_\alpha:U\to\Rext$  be defined as in \cref{eq:multi_regularization}, and let us introduce $H_\alpha^\star \coloneqq \arg\min_{u\in U} H_\alpha(u) $ .
    Then, for Lebesgue-almost all $\alpha \in [0,+\infty)^m$,
    we have: if $H_\alpha^\star \neq \emptyset$, then
    \begin{equation*}
        \inf_{u\in H_\alpha^\star} G_j(u) =
        \sup_{u\in H_\alpha^\star} G_j(u)
        \qquad
        \mbox{for every}
        \qquad
        j=1,\ldots, m.
    \end{equation*}
\end{corollary}
\begin{proof}
    The argument follows the line of the proof of \cref{thm:principle_refined}, i.e., when $H_\alpha^\star \neq \emptyset$, we consider $u^\star \in H_\alpha^\star$ and we apply \cref{thm:multi_reg} to the  minimizing sequence that is constantly equal to $u^\star$.
\end{proof}

Before we prove the measurability of $B_j$, we need to show an ancillary lemma:
\begin{lemma}\label{lemma:inf_measurable}
    Let $H_\alpha:U\to\Rext$   be defined as in \cref{eq:multi_regularization}.
    Then the function $\mathcal{H}\colon [0, +\infty)^m \to \R \cup \{-\infty\}$ defined by
    \begin{equation}  \label{eq:def_value_funct}
        \alpha \mapsto \mathcal{H}(\alpha) \coloneqq \inf_{u \in U} H_{\alpha}(u)
    \end{equation}
    is concave. In particular, it is measurable.
\end{lemma}
\begin{proof}
    Let us consider $\alpha,\alpha'\in[0,+\infty)^m$ and $\theta\in[0,1]$. Then, we compute:
    \begin{equation*}
        \begin{split}
            \mathcal{H}\big(\theta \alpha + (1-\theta)\alpha'\big) & = \inf_{u\in U} \left( F(u) + \sum_{j=1}^m\big(\theta \alpha_j + (1-\theta)\alpha'_j\big) G_j(u) \right) \\
                                                                   & \geq
            \theta\mathcal{H}(\alpha) + (1-\theta)\mathcal{H}(\alpha'). \qedhere
        \end{split}
    \end{equation*}
\end{proof}
We show the the measurability of $B_j$ as part of a stronger result.
We now introduce below some functions that will play a crucial role in our analysis.

\begin{defn} \label{def:G_pm_multi}
    Let $H_\alpha:U\to\Rext$   be defined as in \cref{eq:multi_regularization}.
    For every $\alpha\in [0, +\infty)^m$, we define
    \begin{equation*}
        S_{\alpha} \coloneqq \{(u_i)_{i \in \NN} \text{ minimizing sequence of } H_{\alpha}\},
    \end{equation*}
    and we set for every $j=1,\ldots,m$
    \begin{equation} \label{eq:def_G_pm_multi}
        G_{\alpha}^{j, +}  \coloneqq \sup_{(u_i) \in S_{\alpha}} \limsup_{i \to \infty} G_j(u_i),
        \qquad
        G_{\alpha}^{j, -}  \coloneqq \inf_{(u_i) \in S_{\alpha}} \liminf_{i \to \infty} G_j(u_i).
    \end{equation}
\end{defn}

We now investigate the properties of the functions $\alpha\mapsto G_\alpha^{j,\pm}$.
In what follows, we adopt the following conventions for the $\liminf$ and $\limsup$ of a function $g:[0,+\infty)\to \RextM$:
\begin{equation*}
    \begin{split}
         & \liminf_{\alpha'\to\alpha} g(\alpha') \coloneqq
        \inf \Big\{
        \liminf_{i\to\infty} g(\alpha'_i) \mid \lim_{i\to\infty}\alpha'_i =\alpha
        \,\, \mbox{ and } \,\,
        \alpha'_i\neq \alpha \,\, \forall i
        \Big\},                                            \\
         & \limsup_{\alpha'\to\alpha} g(\alpha') \coloneqq
        \sup \Big\{
        \limsup_{i\to\infty} g(\alpha'_i) \mid \lim_{i\to\infty}\alpha'_i =\alpha
        \,\, \mbox{ and } \,\,
        \alpha'_i\neq \alpha \,\, \forall i
        \Big\}.
    \end{split}
\end{equation*}
It is convenient to recall that every concave function $f$ is continuous in the relative interior of its effective domain, i.e., in the relative interior of the set $\{f>-\infty\}$. See, e.g., \cite[Theorem~10.1]{R97} for the convex setting.

\begin{lemma} \label{lem:finite_lims_G}
    Let $\mathcal{H}\colon [0,+\infty)^m\to \RextM$ be the value function associated to the parametric regularized functionals $H_\alpha\colon U\to \Rext$ with the form as in  \cref{eq:multi_regularization}.
    Let $\mathcal{O}\subset [0,+\infty)^m$ be an open set such that $\val(\alpha) > -\infty$ for every $\alpha\in \mathcal{O}$. Then, for every $j=1,\ldots,m$ we have:
    \begin{equation*}
        -\infty< \liminf_{\alpha'\to\alpha} G_{\alpha'}^{j,-} \leq \limsup_{\alpha'\to\alpha} G_{\alpha'}^{j,+} < +\infty.
    \end{equation*}
    In particular, there exists a neighborhood $\mathcal{O}'\ni\alpha$ with $\mathcal{O}'\subset \mathcal{O}$ such that the functions  $\alpha'\mapsto G_{\alpha'}^{i,\pm}$ are bounded, for every $i=1,\ldots,m$.
\end{lemma}
\begin{proof}
    The inequality $\liminf_{\alpha'\to\alpha} G_{\alpha'}^{j,-} \leq \limsup_{\alpha'\to\alpha} G_{\alpha'}^{j,+}$ follows directly from the fact that, by definition, $G_{\alpha'}^{i,-}\leq G_{\alpha'}^{i,+}$ for every $\alpha'\in[0,+\infty)^m$. Hence, we only have to establish the first and the last bound. \\
    We start by showing that $\limsup_{\alpha'\to\alpha} G_{\alpha'}^{j,+}<+\infty$ for every $j=1,\ldots,m$. Let us assume by contradiction that this is not the case for some $j\in \{1,\ldots,m\}$. Then, we can find sequences $(\alpha_i)_{i\in\NN}$ and $(u_i)_{i\in\NN}\subset U$ such that $\alpha_i\to\alpha$ as $i\to\infty$ and $H_{\alpha_i}(u_i) = \val(\alpha_i) + \e_i $ with $0\leq \e_i\leq 1/i$, and such that $\lim_{i\to\infty} G_j(u_i) =+\infty$. Hence, we choose $h>0$ such that $\overline{B_{2h}(\alpha)}\subset\mathcal{O}$, and we observe that
    \begin{equation*}
        \begin{split}
            \val(\alpha - he_j) & = \lim_{i\to\infty} \val(\alpha_i - he_j) \\
                                & \leq  \lim_{i\to\infty}
            H_{\alpha_i - he_j}(u_i)
            = \lim_{n\to \infty} \left(
            H_{\alpha_i}(u_i) - hG_j(u_i)
            \right)                                                         \\
                                & = \lim_{i\to \infty} \left(
            \val(\alpha_i) + \e_i - hG_j(u_i)
            \right) = \val(\alpha) -  h \lim_{i\to \infty}  G_j(u_i)  = -\infty,
        \end{split}
    \end{equation*}
    contradicting the fact that $\val>-\infty$ in $\mathcal{O}$. We notice that above we used the continuity of the concave function $\val$ at $\alpha,\alpha - he_j\in \mathcal{O}$ (see, e.g., \cite[Theorem~10.1]{R97}).\\
    The proof of the fact that $\liminf_{\alpha'\to\alpha} G_{\alpha'}^{j,-}>-\infty$ for every $j=1,\ldots,m$ follows the lines of the argument presented above. The only difference is that we consider the auxiliary sequence $(\alpha_i+he_j)_{i\in\NN}$ convergent to $\alpha+he_j\in \mathcal{O}$.
    Finally, the last part of the thesis follows automatically.
\end{proof}

The next proposition establishes a regularity result for the mappings $\alpha\mapsto G_\alpha^{j,\pm}$.

\begin{proposition} \label{prop:semicont_G}
    Let $\mathcal{H}\colon [0,+\infty)^m\to \RextM$ be the value function associated to the parametric regularized functionals $H_\alpha\colon U\to \Rext$ with the form as in  \cref{eq:multi_regularization}.
    If $\mathcal{O}\subset [0,+\infty)^m$ is an open set such that $\val(\alpha) > -\infty$ for every $\alpha\in \mathcal{O}$, then
    \begin{equation} \label{eq:semicont_G}
        G_\alpha^{j,+} = \limsup_{\alpha'\to\alpha} G_{\alpha'}^{j,+} = \limsup_{\alpha'\to\alpha} G_{\alpha'}^{j,-} \quad \mbox{ and } \quad
        G_\alpha^{j,-} = \liminf_{\alpha'\to\alpha} G_{\alpha'}^{j,-} = \liminf_{\alpha'\to\alpha} G_{\alpha'}^{j,+}
    \end{equation}
    for every $j=1,\ldots,m$ and for every $\alpha\in \mathcal{O}$. Moreover, we have:
    \begin{equation} \label{eq:side_continuity}
        G_\alpha^{j,+} = \lim_{h\searrow 0} G_{\alpha - he_j}^{j,+} \quad \mbox{ and } \quad
        G_\alpha^{i,-} = \lim_{h\searrow 0} G_{\alpha + he_j}^{j,-}
    \end{equation}
    for every $j=1,\ldots,m$ and for every $\alpha\in \mathcal{O}$.
\end{proposition}
\begin{proof}
    Let us take $\hat{j} \in \{1,\ldots,m\}$ and $\alpha \in \mathcal{O}$. We detail the identity for $G_\alpha^{\hat{j},+}$.
    In virtue of the Monotonicity Lemma for a single regularizer (\cref{lem:monotone_G_relaxed}), we observe that
    \begin{equation} \label{eq:semic_G_part1}
        G_\alpha^{ \hat{j} ,+} \leq \lim_{h\searrow 0} G_{\alpha - he_{\hat{j}}}^{ \hat{j},-}
        \leq \limsup_{\alpha'\to\alpha} G_{\alpha'}^{\hat{j} ,-}
        \leq \limsup_{\alpha'\to\alpha} G_{\alpha'}^{\hat{j} ,+},
    \end{equation}
    where we applied \cref{lem:monotone_G_relaxed} to the regularization parameter $h>0$.
    We are left to prove that $G_\alpha^{\hat{j} ,+}\geq \limsup_{\alpha'\to\alpha} G_{\alpha'}^{ \hat{j} ,+}$.
    Let us consider $(\alpha_i)_{i\in\NN}$ such that $\alpha_i\to\alpha$ as $i\to\infty$ and $\lim_{i\to\infty} G^{\hat{j} ,+}_{\alpha_i} = \limsup_{\alpha'\to\alpha} G_{\alpha'}^{ \hat{j} ,+}$.
    Invoking \cref{lem:finite_lims_G}, we can assume that $(\alpha_i)_{i\in\NN}\subset \mathcal{O}'$, so that there exists $C > 0$ satisfying $\sup_{i\in\NN} |G_{\alpha_i}^{j,\pm}| \leq C$ for every $j=1,\ldots,m$.
    Moreover, we construct $(u_i)_{i\in\NN}\subset U$ such that  $H_{\alpha_i}(u_i) = \val(\alpha_i) + \e_i $ with $0\leq \e_i\leq 1/i$, and such that $|G_{\hat{j}}(u_i) - G^{\hat{j},+}_{\alpha_i}|\leq 1/i$ and $|G_j(u_i)|\leq 2C$ for every $j=1,\ldots,m$.
    We claim that $(u_i)_{i\in\NN}$ is a minimizing sequence for $H_\alpha$.
    Indeed, we compute
    \begin{equation*}
        \begin{split}
            \lim_{i\to\infty} H_\alpha(u_i) & =  \lim_{i\to\infty} \left( H_{\alpha_i}(u_i) + \sum_{j=1}^m(\alpha^j - \alpha^j_i)G_j(u_i) \right)                     \\
                                            & = \lim_{i\to\infty} \left( \val(\alpha_i) + \e_i +  \sum_{j=1}^m(\alpha^j - \alpha^j_i)G_j(u_i) \right) = \val(\alpha),
        \end{split}
    \end{equation*}
    showing that the claim is true, where we used the continuity of the concave value function $\val$ at $\alpha$ and the boundedness of the functions $G_1,\ldots,G_m$ along $(u_i)_{i\geq 1}$.
    Hence, $(u_i)_{i\geq 1}$ being a minimizing sequence for $H_\alpha$, the definition of $G_\alpha^{\hat j,+}$ yields
    \begin{equation*}
        G_\alpha^{ \hat j ,+}\geq \limsup_{i \to\infty} G_{\hat j}(u_i) = \lim_{i \to\infty} G^{ \hat j ,+}_{\alpha_i} = \limsup_{\alpha'\to\alpha} G_{\alpha'}^{ \hat j ,+},
    \end{equation*}
    which is the desired inequality. For the identity in \cref{eq:side_continuity}, it is sufficient to apply \eqref{eq:semicont_G} to the chain of inequalities $G_\alpha^{ \hat{j} ,+} \leq \lim_{h\searrow 0} G_{\alpha - he_{\hat{j}}}^{ \hat{j},+}
        \leq \limsup_{\alpha'\to\alpha} G_{\alpha'}^{\hat{j} ,+}$, which follows from \cref{lem:monotone_G_relaxed}. \\
    The arguments for the identities concerning $G_\alpha^{\hat j ,-}$
    are analogous.
\end{proof}

Finally, we are in position for establishing the measurability of the sets $B_1,\ldots,B_m$.

\begin{corollary} \label{cor:B_measurable}
    The sets $B_1,\ldots,B_m$ defined as in \cref{eq:def_sets_Bj} are Lebesgue-measurable.
\end{corollary}
\begin{proof}
    Recalling that by definition $B_j\subset \{\val > -\infty\}$ (see \cref{eq:def_sets_Bj}), we can partition $B_j$ as follows:
    \begin{equation*}
        B_j = \Big(B_j \cap \partial \{\val > -\infty\} \Big) \cup
        \Big(B_j \cap \mathrm{int } \{\val > -\infty\} \Big),
    \end{equation*}
    where $\partial \{\val > -\infty\}, \mathrm{int } \{\val > -\infty\}$ denote, respectively, the boundary and the interior of the convex set $\{\val > -\infty\}$. Since the boundary of a convex subset of $\R^m$ is closed and its Lebesgue measure $\mathcal{L}_{\R^m}$ is null, we deduce that  $B_j \cap \partial \{\val > -\infty\}$ is Lebesgue measurable (and, in particular, its measure is zero as well).
    On the one hand, in the case $\mathrm{int } \{\val > -\infty\} = \emptyset$, the thesis follows immediately.
    On the other hand, if $\mathrm{int } \{\val > -\infty\} \neq \emptyset$, \cref{prop:semicont_G} guarantees that $\alpha\mapsto G_\alpha^{j,+}$ and $\alpha\mapsto G_\alpha^{j,-}$ are, respectively, upper and lower semi-continuous on $\mathrm{int } \{\val > -\infty\}$. In particular, this implies that the set $B_j \cap \mathrm{int } \{\val > -\infty\} = \{\alpha \mid G_\alpha^{j,+} > G_\alpha^{j,-}\} \cap \mathrm{int } \{\val > -\infty\}$ is measurable.
\end{proof}

We conclude this section with another immediate consequence of \cref{prop:semicont_G}.




\section{Trade-off Invariance Principle for critical points} \label{sec:critical}

In view of \cref{thm:principle}, a natural question is whether the Principle holds true for local minimizers as well.
In this section, we provide a positive answer, and we extend the Trade-off Invariance Principle to the critical points of a regularized functional $H_\alpha$ having the structure as in \cref{eq:def_H}.
To this end, we shall assume that $H_\alpha$ satisfies the celebrated {\L}ojasiewicz inequality (see \cite{lojasiewicz1963propriete}). In the present section, we shall equip the Banach space $U$ exclusively with the \emph{strong topology}.

\begin{theorem} \label{thm:crit_points}
    Let $U$ be a Banach space, and let us consider two functionals $F,G\colon U \to \R$ that are of class $C^1$, i.e., that are everywhere Fr\'echet-differentiable, and whose differentials $\nabla F, \nabla G\colon U\to U'$ are continuous.
    Defining $H_\alpha \coloneqq F + \alpha G$ for $\alpha\in [0,+\infty)$, we assume that $H_\alpha$ satisfies the local {\L}ojasiewicz inequality, i.e., for every $\alpha\in [0,+\infty)$ for every $u\in U$ there exists an open neighborhood $\mathcal{O}_u\ni u$ and constants $C_u>0$ and $\theta_u \in (1,2]$ such that
        \begin{equation}
            |H_\alpha(v) - H_\alpha(u)|\leq  C_u \| \nabla H_\alpha (v)  \|^{\theta_u}
        \end{equation}
        for every $v\in \mathcal{O}_u$.
        Finally, let $K\subset U$ be a strongly compact set.
        Then, for all but countable $\alpha\in [0,+\infty)$, if $u_1,u_2\in K$ are such that $\nabla H_\alpha(u_1) = \nabla H_\alpha(u_2)=0$ and $H_\alpha(u_1) = H_\alpha(u_2)$, then $G_\alpha(u_1) = G_\alpha(u_2)$.
\end{theorem}
\begin{proof}
    We argue by contradiction, and we assume that the thesis is not true.
    In this case,
    the set
    \begin{equation} \label{eq:def_exept_set}
        \begin{split}
            E\coloneqq
            \big\{ \alpha \in [0,+\infty) \mid & \exists u_1,u_2\in K : \nabla H_\alpha(u_1) = \nabla H_\alpha(u_2) = 0, \\
                                               & \qquad
            H_\alpha(u_1) =  H_\alpha(u_2),\,  G(u_1) - G(u_2)>0         \big\}
        \end{split}
    \end{equation}
    consists of uncountably many elements.
    Similarly, if we define
    \begin{equation*}
        \begin{split}
            E^k\coloneqq
            \big\{ \alpha \in [0,+\infty) \mid & \exists u_1,u_2\in K : \nabla H_\alpha(u_1) = \nabla H_\alpha(u_2) = 0, \\
                                               & \qquad
            H_\alpha(u_1) =  H_\alpha(u_2),\,  G(u_1) - G(u_2)>1/k         \big\},
        \end{split}
    \end{equation*}
    we observe that
    \begin{equation*}
        E= \bigcup_{k=1}^\infty E^k,
    \end{equation*}
    and we conclude that there exists $\bar k$ such that $E^k$ contains uncountably many elements.
    Setting $\e\coloneqq 1/\bar k$, we have that there exists a monotone decreasing sequence $(\alpha^i)_{i\in\NN}\subset E^k$ and two sequences $(u_1^i)_{i\in\NN},(u_2^i)_{i\in\NN}$ such that:
    \begin{enumerate}
        \item $A_0\leq \alpha^{i+1}<\alpha^i \leq A_1$ for every $i\in\NN$, and $\alpha^i\to \bar\alpha$ as $i\to\infty$;
        \item $\nabla H_{\alpha^i}(u_1^i) = \nabla H_{\alpha^i}(u_2^i) = 0$;
        \item $H_{\alpha^i}(u_1^i) = H_{\alpha^i}(u_2^i)$;
        \item $G_{\alpha^i}(u_1^i) \geq  G_{\alpha^i}(u_2^i) + \e$.
    \end{enumerate}
    Moreover, leveraging on the compactness of $K$, up to the extraction of a non-relabelled subsequence, we can further assume that there exists $\bar u_1, \bar u_2 \in K$ such that $u_1^i\to \bar u_1$ and $u_2^i\to \bar u_2$ as $i\to \infty$.
    Passing to the limit in (4), we get that $G(\bar u_1)\geq G(\bar u_2) + \e$, which can be rephrased by saying that
    \begin{equation*}
        G(\bar u_1)\geq \bar G + \e/2 \qquad \mbox{and} \qquad G(\bar u_2)\leq \bar G - \e/2,
    \end{equation*}
    where we set $\bar G\coloneqq \frac{G(\bar u_1) + G(\bar u_2)}{2}$.
    In view of finding a contradiction, we aim at showing that for $i$ large enough we have
    \begin{equation} \label{eq:contradict_ineq}
        H_{\alpha^i}(u_1^i) - \alpha^i\bar G > H_{\bar \alpha} (\bar u_1) - \bar\alpha \bar G \qquad
        \mbox{and} \qquad
        H_{\bar \alpha} (\bar u_2) - \bar\alpha \bar G > H_{\alpha^i}(u_2^i) - \alpha^i\bar G,
    \end{equation}
    which, combined with the fact that $H_{\bar \alpha}(\bar u_1) = H_{\bar \alpha}(\bar u_2)$, would result in $H_{\alpha^i}(u_1^i) > H_{\alpha^i}(u_2^i)$, that contradicts with the definition itself of $u_1^i,u_2^i$.
    In order to establish the first relation in \eqref{eq:contradict_ineq},we begin by applying the local {\L}ojasiewicz inequality to the functional $H_{\bar \alpha}$ at the point $\bar u_1$, and we denote with $\mathcal{O}\ni \bar u_1$, with $C_1>0$, and with $\theta_1\in (1,2]$, respectively, the neighborhood of $\bar u_1$, the constant, and the exponent prescribed by this assumption.
    Recalling that $u^i_1\to \bar u_1$ as $i\to\infty$, we deduce that there exists $\bar i \geq 1$ such that we have
    \begin{itemize}
        \item[(a)] $G(u_1^i)\geq \bar G + \e/4$,
        \item[(b)] $|H_{\bar \alpha} (u_1^i) - H_{\bar \alpha} (\bar u_1)|\leq  C_1 \| \nabla H_{\bar \alpha} (v)  \|^{\theta_1}$,
    \end{itemize}
    for every $i\geq \bar i$. Moreover, we recall that, since $K\subset U$ is compact and $\nabla G$ is continuous, there exists $\kappa>0$ such that $\sup_{u\in K} \| \nabla G(u) \|\leq\kappa$.
    Finally, we compute:
    \begin{equation*}
        \begin{split}
            H_{\alpha^i}(u_1^i) - \alpha^i\bar G - \left( H_{\bar \alpha} (\bar u_1) - \bar\alpha \bar G \right)
             & =
            H_{\alpha^i}(u_1^i) - H_{\bar \alpha}(\bar u_1) - (\alpha^i-\bar \alpha ) \bar G                                        \\
             & \geq H_{\alpha^i}(u_1^i) - H_{\bar \alpha}(u_1^i) - (\alpha^i-\bar \alpha ) \bar G
            - \left| H_{\bar \alpha}(u_1^i) - H_{\bar \alpha}(\bar u_1) \right|                                                     \\
             & =  (\alpha^i-\bar \alpha ) (G(u_1^i) -\bar G) - \left| H_{\bar \alpha}(u_1^i) - H_{\bar \alpha}(\bar u_1) \right|    \\
             & \geq \frac\e4 (\alpha^i-\bar \alpha ) - C_1 \left\| \nabla F(u_1^i) + \bar\alpha \nabla G(u_1^i) \right\|^{\theta_1} \\
             & = \frac\e4 (\alpha^i-\bar \alpha ) - C_1 \left\|  (\bar\alpha - \alpha^i) \nabla G(u_1^i) \right\|^{\theta_1}        \\
             & \geq \frac\e4 (\alpha^i-\bar \alpha ) - C_1 \kappa^{\theta_1} |\bar\alpha - \alpha^i|^{\theta_1},
        \end{split}
    \end{equation*}
    where we used that, by construction, $\nabla H_{\alpha^i}(u_1^i) = \nabla F (u_1^i) + \alpha^i  \nabla G (u_1^i)=0$. Recalling that $\theta_1\in (1,2]$ and that $\alpha^i-\bar \alpha\searrow 0$ as $i\to\infty$, we obtain the first inequality of \eqref{eq:contradict_ineq}. The second relation follows by repeating \emph{verbatim} the same argument to the difference  $H_{\bar \alpha} (\bar u_2) - \bar\alpha \bar G  - \left(  H_{\alpha^i}(u_2^i) - \alpha^i\bar G \right)$, and by noting that $G(u_2^i)\leq \bar G  -\e/4$ for $i$ sufficiently large.
\end{proof}

\begin{remark}
    In the case $U$ is $\sigma$-compact, i.e., if $U$ is finite-dimensional, the statement of \cref{thm:crit_points} is true even if we allow the critical points $u_1,u_2$ to be chosen in the whole domain $U$.
    More in general, we can avoid the restriction to a compact set $K\subset U$ if there exists a $\sigma$-compact subset of $U$ that contains every critical point of $H_\alpha$ for every $\alpha \in [0,+\infty)$.
\end{remark}



\section{Consequences and applications} \label{sec:consequences}

In order to exemplify and highlight the relevance of the results stated in  \cref{sec:intro}, we collect below some of their more immediate implications.

\subsection{Differentiability points of the value function} \label{subsec:diff_value}
In this paragraph, we relate the Trade-off Invariance Principle to the points where the value function $\mathcal{H}$ (see \cref{eq:def_value_funct}) is differentiable.
This provides a complete generalization of the results obtained in \cite[Section~2]{ito2011multiparameter}. Namely, in their work, the authors equipped the domain with a topological structure, and they further assume $F,G_1,\ldots,G_m$ to be coercive and lower semi-continuous. In this way, their arguments could rely on the existence of minimizers of $H_\alpha = F + \sum_{j=1}^m\alpha_jG_j$ for every $\alpha = (\alpha_1,\ldots,\alpha_m) \in [0,+\infty)^m$.  In addition, in \cite{ito2011multiparameter} the regularizers $G_1,\ldots,G_m$ are required to be non-negative.
As we are going to show below, the intertwining between the differentiability of the value function $\val$ and the Trade-off Invariance Principle holds more generally, and it does not require at all the existence of minimizers for the underlying regularized problems.

Let $\mathcal{O}\subset [0,+\infty)^m$ be an open set such that $\val(\alpha) > -\infty$ for every $\alpha\in \mathcal{O}$.
For every $\alpha\in \mathcal{O}$ we set the following notations:
\begin{equation*}
    \partial^+_j \val (\alpha) \coloneqq
    \lim_{h\to 0^+} \frac{\val(\alpha + he_j) - \val(\alpha)}{h},
    \qquad
    \partial^-_j \val (\alpha) \coloneqq
    \lim_{h\to 0^+} \frac{\val(\alpha - he_j) - \val(\alpha)}{-h}
\end{equation*}
for every $j=1,\ldots,m$, where $e_j$ is the $j$-th element of the standard basis of $\R^m$.
We recall that the existence of the limits follows directly from the concavity of $\val$ (see \cref{lemma:inf_measurable}).


\begin{lemma} \label{lem:ident_derivatives}
    Let $\mathcal{H}\colon [0,+\infty)^m\to \RextM$ be the value function associated to the parametric regularized functionals $H_\alpha\colon U\to \Rext$ with the form as in  \cref{eq:multi_regularization}.
    Moreover, for every $i=1,\ldots,m$ let $G_\alpha^{j,+},G_\alpha^{j,-}$ be as in \cref{def:G_pm_multi}.
    Finally, let $\mathcal{O}\subset [0,+\infty)^m$ be an open set such that $\val(\alpha) > -\infty$ for every $\alpha\in \mathcal{O}$.
    Then, for every $j=1,\ldots,m$ and for every $\alpha\in \mathcal{O}$, we have
    \begin{equation*}
        \partial^+_j \mathcal{H} (\alpha)
        = G_\alpha^{j,-}
        \leq G_\alpha^{j,+}
        = \partial^-_j \mathcal{H} (\alpha).
    \end{equation*}
\end{lemma}
\begin{proof}
    Let us start by showing that $
        \partial^-_j \mathcal{H} (\alpha) \geq G_\alpha^{j,+}$ for every $j=1,\ldots,m$.
    Let us take $j\in\{1,\ldots,m\}$ and $\alpha\in [0,+\infty)^m$ such that $\mathcal{H}(\alpha)>-\infty$.
    For every $u \in U$, we have
    \begin{equation*}
        \begin{split}
            \val(\alpha) - \val(\alpha-he_j) & \geq H_\alpha(u) + (\val(\alpha) - H_\alpha(u)) - H_{\alpha-he_j}(u) \\
                                             & = h G_j(u) + (\val(\alpha) - H_\alpha(u)).
        \end{split}
    \end{equation*}
    Let us consider a monotonically decreasing sequence $(h_i)_{i\in\NN}$ such that $h_i\searrow 0$, and let us take a minimizing sequence $(u_i)_{i\in\NN}$ for $H_\alpha$ such that $\lim_{i\to \infty}G_j(u_i)=G^{j,+}_\alpha$. Moreover, up to the extraction of a subsequence, we can assume that $H_\alpha(u_i)-\val(\alpha)\leq h_i^2$.
    If we compute $\partial_j^-\val(\alpha)$ along $(h_i)_{i\in\NN}$, from the previous estimate we deduce that
    \begin{equation*}
        \partial_j^-\val(\alpha) = \lim_{i\to \infty}
        \frac{\val(\alpha-h_ie_j) - \val(\alpha)}{-h_i}
        \geq \lim_{i\to \infty}G_j(u_i)=G^{j,+}_\alpha.
    \end{equation*}
    We now address $\partial^-_j \mathcal{H} (\alpha) \leq G_\alpha^{j,+}$ for every $j=1,\ldots,m$.
    As before, let us consider a monotonically decreasing sequence $(h_i)_{i\in\NN}$ such that $h_i\searrow 0$, and let us introduce the notation $\alpha_i\coloneqq \alpha - h_ie_j$. Moreover, we construct two sequences $(u_i)_{i\in\NN},(u_i')_{i\in\NN}\subset U$ such that
    \begin{equation*}
        H_{\alpha}(u_i) = \val(\alpha) + \e_i, \quad
        |G_j(u_i) - G_\alpha^{j,+}|\leq\frac1i,
    \end{equation*}
    and
    \begin{equation*}
        H_{\alpha_i}(u_i') = \val(\alpha_i) + \e_i', \quad
        |G_j(u_i') - G_{\alpha_i}^{j,+}| \leq\frac1i,
    \end{equation*}
    with $0\leq\e_i,\e_i'\leq h_i^2$ for every $i\in\NN$.
    From the fact that $\val(\alpha)\leq H_\alpha(u'_i)$, we deduce that
    \begin{equation} \label{eq:dist_H_alpha}
        H_\alpha(u_i) - H_\alpha(u'_i) \leq \e_i
    \end{equation}
    for every $i\in\NN$.
    Then, we compute
    \begin{equation} \label{eq:bound_der_2}
        \begin{split}
            \val(\alpha) - \val(\alpha_i) & =
            H_\alpha(u_i) - H_{\alpha_i}(u_i') -\e_i + \e_i'                                                                                     \\
                                          & = h_i G_j(u_i) + h_i \big( G_j(u'_i) - G_j(u_i) \big) + H_\alpha(u_i) - H_\alpha(u_i') -\e_i + \e_i' \\
                                          & \leq h_i G_j(u_i) + h_i \big( G_j(u'_i) - G_j(u_i) \big) + \e_i'                                     \\
                                          & = h_i G_j(u_i)                                                                                       \\
                                          & \qquad + h_i
            \left[ \big( G_{\alpha_i}^{j,+} - G_{\alpha}^{j,+} \big) + \big( G_j(u'_i) - G_{\alpha_i}^{j,+} \big) + \big( G_{\alpha}^{j,+} - G_j(u_i) \big)
                \right],
        \end{split}
    \end{equation}
    where we used \cref{eq:dist_H_alpha} to get the inequality.
    Recalling that $\alpha_i=\alpha-h_ie_j$, \cref{eq:side_continuity} implies that $G_{\alpha_i}^{j,+} - G_{\alpha}^{j,+}\to 0$ as $i\to \infty$. Therefore, if we divide by $h_i$ and we pass to the limit in \cref{eq:bound_der_2}, we obtain that
    \begin{equation*}
        \partial_j^-\val(\alpha) = \lim_{i\to \infty}
        \frac{\val(\alpha-h_ie_j) - \val(\alpha)}{-h_i}
        \leq \lim_{i\to \infty}G_j(u_i)=G^{j,+}_\alpha.
    \end{equation*}
    The identity $\partial^+_j \mathcal{H} (\alpha) = G_\alpha^{j,-}$ follows from the same arguments.
\end{proof}

We illustrate in the next result the relation between the differentiability of $\val$ and the Trade-off Invariance Principle.

\begin{proposition} \label{prop:different_value}
    Let $\mathcal{H}\colon [0,+\infty)^m\to \RextM$ be the value function associated to the parametric regularized functionals $H_\alpha\colon U\to \Rext$ with the form as in  \cref{eq:multi_regularization}.
    Moreover, for every $i=1,\ldots,m$ let $G_\alpha^{j,+},G_\alpha^{j,-}$ be as in \cref{def:G_pm_multi}.
    Finally, let $\mathcal{O}\subset [0,+\infty)^m$ be an open set such that $\val(\alpha) > -\infty$ for every $\alpha\in \mathcal{O}$.
    Then, for every $j=1,\ldots,m$ and every $\alpha\in \mathcal{O}$ the following are equivalent:
    \begin{enumerate}
        \renewcommand{\theenumi}{(\arabic{enumi})}
        \renewcommand{\labelenumi}{\theenumi}
        \item The partial derivative $\partial_j\val(\alpha)$ exists, i.e., $\partial_j^+\val(\alpha) = \partial_j^-\val(\alpha)$; \label{case:partial_deriv}
        \item The identity $G_\alpha^{j,+}=G_\alpha^{j,-}$ holds; \label{case:identity}
        \item The function $\alpha'\mapsto G_{\alpha'}^{j,+}$ is continuous at $\alpha$; \label{case:continuity_G+}
        \item The function $\alpha'\mapsto G_{\alpha'}^{j,-}$ is continuous at $\alpha$. \label{case:continuity_G-}
    \end{enumerate}
    Finally, the value function $\val$ is differentiable at $\alpha\in\mathcal{O}$ if and only if $G_\alpha^{j,+}=G_\alpha^{j,-}$ for every $j=1,\ldots,m$.
\end{proposition}
\begin{proof}
    The equivalence \ref{case:partial_deriv}$\iff$\ref{case:identity} follows directly from \cref{lem:ident_derivatives}.
    To establish \ref{case:identity}$\implies$\ref{case:continuity_G+}, we use \cref{prop:semicont_G} to see that
    \begin{equation*}
        G_\alpha^{j,+} = \limsup_{\alpha'\to\alpha} G_{\alpha'}^{j,+}\geq \liminf_{\alpha'\to\alpha} G_{\alpha'}^{j,+} \geq \liminf_{\alpha'\to\alpha}  G_{\alpha'}^{j,-} =  G_{\alpha}^{j,-}.
    \end{equation*}
    Therefore, from the identity in \ref{case:identity}, it follows that $G_\alpha^{j,+} = \lim_{\alpha'\to\alpha} G_{\alpha'}^{j,+}$.
    Concerning the inverse implication \ref{case:continuity_G+}$\implies$\ref{case:identity}, the continuity of $\alpha'\mapsto G^{j,+}_{\alpha'}$ in \ref{case:continuity_G+} implies in particular $G^{j,+}_{\alpha}=\lim_{h\searrow 0}G^{j,+}_{\alpha + he_j}$. Moreover, the Monotonicity Lemma (see \cref{lem:monotone_G_relaxed}) applied to the one-dimensional regularization parameter $h>0$ yields
    \begin{equation*}
        G^{j,+}_{\alpha} \geq G^{j,-}_{\alpha} \geq G^{j,+}_{\alpha + h e_j},
    \end{equation*}
    so that, when letting $h\searrow 0$, we get $G^{j,+}_{\alpha} = G^{j,-}_{\alpha}$. This concludes the equivalence \ref{case:identity}$\iff$\ref{case:continuity_G+}.
    The bi-implication \ref{case:identity}$\iff$\ref{case:continuity_G-} follows from similar arguments. \\
    Now we address the last part of the thesis.
    Since the differentiability of $\val$ at $\alpha$ implies the existence of the partial derivatives $\partial_j\val(\alpha)$ for every $j=1,\ldots,m$, we immediately deduce that $G_\alpha^{j,+}=G_\alpha^{j,-}$ for every $j=1,\ldots,m$.
    For the inverse implication, let us assume that $\val$ is not differentiable at $\alpha$.
    Denoting with $\nabla \val(\alpha)\subset \R^m$ the super-gradient of the concave function $\val$ at $\alpha$, since $\alpha\in \mathcal{O}$, it follows that $\nabla \val(\alpha) \neq \emptyset$.
    Moreover, we recall that the differentiability of $\val$ at $\alpha$ holds if and only if $\nabla \val(\alpha)$ contains a single element (see, e.g., \cite[Theorem~25.1]{R97} for the convex setting).
    Hence, if $\val$ is not differentiable at $\alpha$, we can find $(p_1,\ldots,p_m),(p'_1,\ldots,p_m')\in \nabla \val(\alpha)$ such that there exists $j\in\{1,\ldots,m\}$ for which $p_j<p_j'$. Recalling the definition of super-gradient, this implies that
    \begin{equation*}
        \frac{\val(\alpha + he_j) - \val(\alpha)}{h} \leq
        p_j < p_j'
        \leq
        \frac{\val(\alpha - he_j) - \val(\alpha)}{-h}
    \end{equation*}
    for every $h>0$. Passing to the limit as $h\to 0$ in the previous expression, we get that $\partial_j^-\val(\alpha)>\partial_j^+\val(\alpha)$, which in turn implies that $G_\alpha^{j,+}>G_\alpha^{j,-}$ (see \cref{lem:ident_derivatives}).
\end{proof}

\subsection{Penalty functions in equality constrained optimization}
Let us consider the following constrained minimization problem:
\begin{equation} \label{eq:probl_constr}
    \min_{u\in U} F(u) \quad \mbox{subject to} \quad G(u)=0,
\end{equation}
where $F, G\colon U\to\Rext$. Up to substituting $G$ with $G^2$, it is not restrictive to assume that $G\geq 0$. In many situations, it is convenient to relax the hard-constrained problem reported in \cref{eq:probl_constr} by incorporating $G$ in the objective function as a penalty term for the violation of the constraint:
\begin{equation} \label{eq:probl_relax}
    \min_{u\in U} F(u) + \alpha G(u),
\end{equation}
with $\alpha\in (0,+\infty)$.
Owing to \cref{thm:principle_refined}, we obtain that, excluding at most countably many values of the penalization parameter $\alpha\in (0,+\infty)$, whenever different minimizers $u^\star_1\neq u^\star_2$ of the relaxed reformulation~\eqref{eq:probl_relax} exist, they are `equivalent' in terms of the constraint, i.e.~$G(u^\star_1)=G(u^\star_2)$.\\
As a relevant example, let us consider a control system in $\R^n$ of the form
\begin{equation*}
    \begin{cases}
        \dot x_u(t) = h(x_u(t),u(t)) & \mbox{for a.e. }t\in[0,T], \\
        x_u(0)=x_0,
    \end{cases}
\end{equation*}
where $h\colon \R^n\times \R^m\to \R^n$ is a smooth map such that $|h(x,v)|\leq C(1+|x|)(1+|v|^p)$ for every $x\in\R^n$ and $v\in\R^m$ ($1\leq p<\infty$), and $u \in U\coloneqq L^p([0,T],\R^m)$.
Moreover, let us consider the endpoint-constrained problem
\begin{equation}\label{eq:probl_constr_ctrl}
    \min_{u\in U} \int_0^T f(x_u(t),u(t)) \,\mathrm{d}t \quad \mbox{subject to} \quad x_u(T) = x_T
\end{equation}
where $f\colon \R^n\times \R^m\to \R$ is a continuous and bounded below function that prescribes the running cost, and $F(u)\coloneqq \int_0^T f(x_u(t),u(t)) \,dt$.
Let us relax the hard terminal-state constraint in \cref{eq:probl_constr_ctrl} by formulating the free-endpoint problem:
\begin{equation} \label{eq:probl_relax_ctrl}
    \min_{u\in U} \int_0^T f(x_u(t),u(t)) \, \mathrm{d}t + \alpha G(x_u(T))
\end{equation}
where $G\colon \R^n\to\R$ is such that $G\geq 0$ and $\{G=0\}=\{x_T\}$, and $\alpha\in(0,+\infty)$ tunes the penalty term.
Excluding at most countably many exceptional values of $\alpha$, if $u^\star_1, u^\star_2$ are optimal controls for the relaxed free-end point problem~\eqref{eq:probl_relax_ctrl}, it surprisingly turns out that $G(x_{u_1^\star}(T))= G(x_{u_2^\star}(T))$.

\subsection{Regularized functionals in Sobolev Spaces}

\cref{cor:strong_convergence} finds application to functionals that are defined on Sobolev spaces $W^{1,p}(\Omega,\R^m)$ with $1<p<\infty$ and $\Omega\subset\R^n$ open, and that are regularized with Sobolev semi-norms.
Namely, let us set $U\coloneqq W^{1,p}(\Omega,\R^m)$, and let us consider a functional $H_\alpha \colon U\to\Rext$ of the form
\begin{equation*}
    H_\alpha(u)\coloneqq F(u) + \alpha g(\|Du\|_{L^p}),
\end{equation*}
where $\alpha\in(0,+\infty)$, $Du\in L^p$ denotes the weak derivative of $u$, and $g\colon [0,+\infty)\to[0,+\infty)$ is a strictly monotone increasing and continuous function. Finally, we assume that $F\not \equiv +\infty$.
We recall that the weak convergence in $W^{1,p}$ ($1<p<\infty$) is characterized as follows: a sequence $(u_i)_{i\in\NN}\subset W^{1,p}$ is $W^{1,p}$-weakly convergent to $u_\infty$ if and only if $(u_i)_{i\in\NN}$ is $L^p$-strongly convergent to $u_\infty$ and $(D u_i)_{i\in\NN}$ is $L^p$-weakly convergent to $D u_\infty$.
Setting $G(u)\coloneqq g(\|Du\|_{L^p})$, it turns out that $G$ is a weakly-to-strong functional for $W^{1,p}$-weakly convergent sequences (see \cref{def:weak2strong}).
Then, by virtue of \cref{cor:strong_convergence}, we deduce that, with the exception of at most countably many $\alpha\in(0,+\infty)$, for a minimizing sequence $(u_i)_{i\in\NN}\subset W^{1,p}$, the properties of being weakly and strongly convergent to a minimizer are equivalent.

\begin{remark}
    The argument reported above can also be adapted to functionals of the form $H_\alpha(u)=F(u) + \alpha g\big (|Du|(\Omega) \big)$ defined over $BV(\Omega,\R^m)$, where $\Omega\subset \R^n$ is open, $g\colon [0,+\infty)\to[0,+\infty)$ is a strictly monotone increasing and continuous function, and $|Du|(\Omega)$ denotes the total variation of $u$. In this case, excluding at most countably many values of $\alpha$, if $(u_i)_{i\in\NN}$ is a minimizing sequence for $H_\alpha$ that is weakly$^*$ converging to a minimizer $u^\star$, then  $(u_i)_{i\in\NN}$ is actually \emph{strictly} converging to $u^\star$ (see~\cite[Section~3.1]{BV} for these notions of convergence).
\end{remark}

\subsection{Approximability through sequentially weakly pre-compact subsets}
We present here a negative result that is implied by \cref{cor:strong_convergence}.

\begin{proposition} \label{prop:negative}
    Let $U$ be a  Banach space and $\mathcal{D} \subseteq U$ be any subset. Let us consider $F\colon U \to \Rext$ such that $F\not\equiv +\infty$ and $G\colon U \to \R$ a \emph{weak-to-strong} functional, and let $H_\alpha\coloneqq F + \alpha G$ be the regularized functional, with $\alpha\in(0,+\infty)$.
    Then, for all but countable $\alpha \in (0, +\infty)$, the following holds: \\
    If $u^\star \in H^\star_\alpha$ is such that $\inf_{v\in \mathcal{D}} \| u^\star- v\|_U >0$, then it is not possible to find a sequence $(u_i)_{i\in\NN}\subset \mathcal{D}$ that is weakly converging to $u^\star$ and that is a minimizing sequence for $H_\alpha$.
    Moreover, if $H_\alpha$ is sequentially weakly lower semi-continuous, if
    $H^\star_\alpha$ is contained in the sequentially weak closure of $\mathcal{D}$
    and
    \begin{equation}\label{eq:D_far_away}
        \inf_{v\in \mathcal{D}} \| u^\star- v\|_U >0 \quad \forall u^\star \in H_\alpha^\star,
    \end{equation}
    then it is not possible to find a sequence $(u_i)_{i\in\NN}\subset \mathcal{D}$ that is weakly pre-compact and that is a minimizing sequence for $H_\alpha$.
    In particular, if $\mathcal{D}$ is sequentially weakly pre-compact, there is no minimizing sequence for $H_\alpha$ in $\mathcal{D}$.
\end{proposition}
As an example of where the above situation occurs, let us consider $I \subseteq \R$ an interval, $U \coloneqq L^p(I)$ with $1<p<\infty$, $V \subseteq \R$ with $\# V<\infty$, and let us denote with $\mathcal{D} \subseteq U$ the set of
functions on $I$ taking values in $V$. Then $\mathcal{D}$ is sequentially weakly pre-compact, and it turns out that its $L^p$-weak closure is $\overline{\mathcal{D}}^{w} =\{u\in L^p(I): \min V\leq u \leq \max V \, \mbox{a.e.} \}$.
Furthermore, when dealing with a functional $H_\alpha = F + \alpha G$, one in general expects its minimizers to take infinitely many values, so that \cref{eq:D_far_away} holds true, since
\begin{equation*} 
    \inf_{v \in \mathcal{D}} \|u - v\|_{L^p}^p = \int_{I} \inf_{v \in V}|u(x) - v|^p \, \mathrm{d}x
\end{equation*}
for every $u\in U$.
By \cref{prop:negative}, in such a scenario there cannot be a minimizing sequence $(u_i)_{i\in\NN}\subset \mathcal{D}$ for $H_\alpha$,  for all but countable $\alpha \in (0, +\infty)$.

\subsection{\texorpdfstring{$\Gamma$}{Gamma}-convergence with respect to the weak topology}

We shall show how~\cref{cor:strong_convergence} affects the sequences obtained by considering quasi-minimizers of $\Gamma$-convergent functionals that contain a weak-to-strong regularization term.
For a thorough introduction to $\Gamma$-convergence, we refer to the textbook~\cite{D93}.
Here, we assume that $U$ is a reflexive Banach space with a separable dual $U'$.
In this framework, following~\cite[Chapter~8]{D93}, given $H\colon U\to \Rext$ and $H^i\colon U\to \Rext$ for $i\in\NN$, we can characterize the $\Gamma$-convergence with respect to the weak topology as follows: The sequence $(H^i)_{i\in\NN}$ is $\Gamma$-converging to $H$ if the following two conditions hold:
\begin{itemize}
    \item ($\liminf$ inequality) For every $u\in U$ and for every $(u_i)_{i\in\NN}$ that is weakly convergent to $u$ we have
          \begin{equation*}
              H(u) \leq \liminf_{i\to\infty} H^i(u_i).
          \end{equation*}
    \item ($\limsup$ inequality) For every $u\in U$, there exists a sequence $(u_i)_{i\in\NN}$ that is weakly convergent to $u$ and such that
          \begin{equation*}
              H(u) \geq \limsup_{i\to\infty} H^i(u_i).
          \end{equation*}
\end{itemize}
Finally, we recall that the functionals $(H^i)_{i\in\NN}$ are \emph{weakly equi-coercive} if, for every $t\in\R$, there exists $K\subset U$ weakly compact such that $\{ u\in U: H^i(u)\leq t\}\subset K$ for every $i\in\NN$.

\begin{proposition} \label{prop:G-conv}
    Let $U$ be a reflexive Banach space with separable dual $U'$.
    Let us consider $F\colon U \to \Rext$ such that $F\not\equiv +\infty$ and $G\colon U \to \R$ a \emph{weak-to-strong} functional, and let $H_\alpha\coloneqq F + \alpha G$ be the regularized functional, with $\alpha\in(0,+\infty)$.
    Then, for all but countable $\alpha \in (0, +\infty)$, the following holds: \\   Let $(H_\alpha^i)_{i\in\NN}$ be a sequence of weakly equi-coercive functionals that are $\Gamma$-convergent to $H_\alpha$ with respect to the weak convergence of $U$, and let $(u_i)_{i\in\NN}$ be a sequence in $U$ that satisfies
    \begin{equation} \label{eq:quasi-minim}
        H_\alpha^i(u_i)\leq \inf_{U} H_\alpha^i+ \e_i \quad \forall \, i\in\NN  \qquad \mbox{with }\qquad
        \lim_{i\to\infty}\e_i=0.
    \end{equation}
    If $(u_i)_{i\in\NN}$ \emph{is a minimizing sequence for $H_\alpha$}, then $(u_i)_{i\in\NN}$ is \emph{strongly pre-compact}.
\end{proposition}
\begin{proof}
    The thesis follows from \cref{cor:strong_convergence}, and from the fact that $(u_i)_{i\in\NN}$ is weakly pre-compact and that any limiting point is a minimizer for $H_\alpha$.
    The latter is a well-established result (see~\cite[Chapter~8]{D93}), but we prove it here for completeness.
    Since $H_\alpha\not\equiv +\infty$, there exists $\hat u\in U$ such that
    \begin{equation*}
        +\infty> H_\alpha(\hat u) \geq \limsup_{i\to\infty} H^i_\alpha(\hat u_i) \geq \limsup_{i\to\infty} \inf_U H^i_\alpha,
    \end{equation*}
    where $(\hat u_i)_{i\in\NN}$ is the sequence prescribed for $\hat u$ by the $\limsup$ condition. Therefore, by virtue of the equi-coercivity of the converging functionals, we deduce that there exists a weakly compact set $K$ that contains the sequence $(u_i)_{i\in\NN}$ defined through \cref{eq:quasi-minim}. Then, let us extract a subsequence $(u_{i_j})_{j\in\NN}$ that is weakly converging to a point $u^\star$. We want to show that $H_\alpha(u^\star)=\inf_U H_\alpha$.
    To see that, let us consider any $\hat u\in U$, and let $(\hat u_i)_{i\in\NN}$ be the sequence prescribed for $\hat u$ by the $\limsup$ condition.
    Hence, we have:
    \begin{equation*}
        H_\alpha(\hat u) \geq \limsup_{j\to\infty} H_\alpha^{i_j}(\hat u_{i_j})
        \geq \limsup_{j\to\infty} \inf_U H_\alpha^{i_j}
        = \limsup_{j\to\infty} H_\alpha^{i_j}(u_{i_j})
        \geq \liminf_{j\to\infty} H_\alpha^{i_j}(u_{i_j})
        \geq H_\alpha (u^\star),
    \end{equation*}
    where we used \cref{eq:quasi-minim} in the equality, and the $\liminf$ condition for $u^\star$ in the last inequality. Since $H_\alpha(\hat u)\geq H_\alpha(u^\star)$ for every $\hat u \in U$, we deduce that $u^\star$ is a minimizer.
\end{proof}



\subsection*{Acknowledgements}
The authors want to sincerely thank G.~Auricchio, M.~Viola and M.~Tamburro for the helpful comments on the early versions of this manuscript.
A.S.~acknowledges partial support from INdAM--GNAMPA.

\bibliographystyle{abbrv}
\bibliography{biblio}

\end{document}